\documentclass[12pt]{article}
\usepackage{}
\usepackage[numbers,sort&compress]{natbib}

\usepackage{color}
\usepackage{threeparttable}

\usepackage[english]{babel}
\usepackage{t1enc}
\usepackage{amsthm,amsmath,amssymb}
\usepackage{mathrsfs}
\usepackage[mathscr]{eucal}
\usepackage[latin1]{inputenc}
\usepackage[english]{babel}
\usepackage{latexsym}
\usepackage{graphicx}
\usepackage[noend]{algpseudocode}
\usepackage{algorithmicx,algorithm}
\usepackage{lineno}
\usepackage{booktabs}
\usepackage{multirow}
\newtheorem{theorem}{Theorem}
\newtheorem{corollary}{Corollary}

\newtheorem{lemma}{Lemma}

\newtheorem{observation}{Observation}
\newtheorem{conjecture}{Conjecture}

\newtheorem{definition}{Definition}

\def\sqrt{\textup{sqrt}}
\def \T{\textup{T}}

\def \diag{\textup{diag~}}

\def \Res{\textup{Res}}
\makeatletter
\newcommand{\rmnum}[1]{\romannumeral #1}
\newcommand\restr[2]{{
		\left.\kern-\nulldelimiterspace 
		#1 
		\right|_{#2} 
}}
\newcommand{\Rmnum}[1]{\expandafter\@slowromancap\romannumeral #1@}
 
\makeatother
\title{Proof of a conjecture on the determinant of walk matrix of rooted product with a  path}
\author{\small Wei Wang$^{{\rm a}}$\quad\quad Zhidan Yan$^{\rm a}$\thanks{Corresponding author: yanzhidan.math@gmail.com}\quad\quad Lihuan Mao$^{\rm b}$
\\
{\footnotesize$^{\rm a}$School of Mathematics, Physics and Finance, Anhui Polytechnic University, Wuhu 241000, China}\\
{\footnotesize$^{\rm b}$School of Mathematics and Data Science, Shaanxi University of Science and Technology, Xi'an 710021, China}
}
\date{}
\begin{document}
 \maketitle

\begin{abstract}
 The walk matrix  of an $n$-vertex graph $G$ with adjacency matrix $A$, denoted by $W(G)$, is $[e,Ae,\ldots,A^{n-1}e]$, where $e$ is the all-ones vector. Let $G\circ P_m$ be the rooted product of $G$ and a rooted path $P_m$ (taking an endvertex as the root), i.e., $G\circ P_m$ is a graph obtained from $G$ and $n$ copies of $P_m$ by identifying each vertex of $G$ with an endvertex of a copy of $P_m$.  Mao-Liu-Wang (2015) and Mao-Wang (2022) proved that, for $m=2$ and $m\in\{3,4\}$, respectively
 $$\det W(G\circ P_m)=\pm a_0^{\lfloor\frac{m}{2}\rfloor}(\det W(G))^m,$$
  where $a_0$ is the constant term of the characteristic polynomial of $G$. Furthermore, Mao-Wang (2022) conjectured that the formula holds for any $m\ge 2$. In this note, we verify this conjecture using the
  technique of Chebyshev polynomials.
 \\


\end{abstract}
\section{Introduction}
\label{intro}
Let $G$ be a simple graph with vertex set $\{1,\ldots,n\}$.  The \emph{adjacency matrix} of $G$ is the $n\times n$ symmetric matrix $A=(a_{i,j})$, where $a_{i,j}=1$ if $i$ and $j$ are adjacent;  $a_{i,j}=0$ otherwise.  For a graph $G$, the \emph{walk matrix} of $G$ is
\begin{equation}
W=W(G):=[e,Ae,\ldots,A^{n-1}e],
\end{equation}
where $e$ is the all-ones vector. Note that walk matrices are clearly integral but are usually not symmetric.  Compared to general integral matrices, walk matrices of graphs have some special properties. For example, the determinant of any walk matrix of order $n$ is always a multiple of $2^{\lfloor\frac{n}{2}\rfloor}$. This kind of  matrices has  attracted  increasing attention in recent years as many interesting properties of graphs are closely related to the corresponding walk matrices. A typical result is a theorem of Wang \cite{wang2017JCTB} which says that any graph $G$ with $2^{-\lfloor\frac{n}{2}\rfloor}\det W(G)$ odd and square-free is uniquely determined (up to isomorphism) by its generalized spectrum. Here, the generalized spectrum of a graph $G$ means the spectrum of $G$ together with that of its complement $\overline{G}$.

In \cite{mao2015}, in order to construct graphs that are determined by generalized spectrum, the authors considered the rooted product graph $G\circ P_2$ and proved the following theorem. Throughout this paper, we shall fix a graph $G$ and use $a_0$ to denote the constant term of the characteristic polynomial of $G$.
\begin{theorem}[\cite{mao2015}]
$$\det W(G\circ P_2)=\pm a_0(\det W(G))^2.$$
\end{theorem}
Recently, Mao and Wang \cite{mao2022} extended the above formula to  $G\circ P_m$ for  $m=3,4$. Precisely, they proved:
\begin{theorem}[\cite{mao2022}]For $m=3$ or $4$,
	$$\det W(G\circ P_m)=\pm a_0^{\lfloor\frac{m}{2}\rfloor}(\det W(G))^m.$$
\end{theorem}
In the same paper, the authors proposed the following natural conjecture, which (if true) unifies and extends the above two theorems.
\begin{conjecture}[\cite{mao2022}]\label{mainconj}
	For any positive integer $m\ge 2$,
		$$\det W(G\circ P_m)=\pm a_0^{\lfloor\frac{m}{2}\rfloor}(\det W(G))^m.$$
\end{conjecture}
The main aim of this note is to confirm this conjecture.
\begin{theorem}\label{main}
	Conjecture \ref{mainconj} is true.
\end{theorem}
The overall strategy of the proof is the same as in \cite{mao2022} which is based on explicit computations of eigenvalues and eigenvectors. Indeed, most derivations in \cite{mao2022} for $m=3,4$  can be easily extended to any positive $m$ except some resultant-related computations.    A new finding of this note is that most computations involved are closely related to Chebyshev polynomials of the second kind. A crucial step is a newly established equality concerning the resultant of two special linear combinations of Chebyshev polynomials (Lemma \ref{newres}).
\section{Eigenvalues and eigenvectors of $G\circ P_m$}
We always regard the endvertex of $P_m$ as its root vertex. For an $n$-vertex labelled graph $G$, the \emph{rooted product graph} $G\circ P_m$, is a graph obtained from $G$ and $n$ copies of $P_m$ by identifying the  root vertex of the $i$-th copy of $P_m$ with the $i$-th vertex of $G$ for $i=1,2,\ldots,n$, see Fig.~1 for an illustration. This is a special case of rooted product graphs $G\circ H$ introduced by Godsil-McKay \cite{godsil1978} and Schwenk \cite{schwenk1974}.
\begin{figure}
	\centering
	\includegraphics[height=6.6cm]{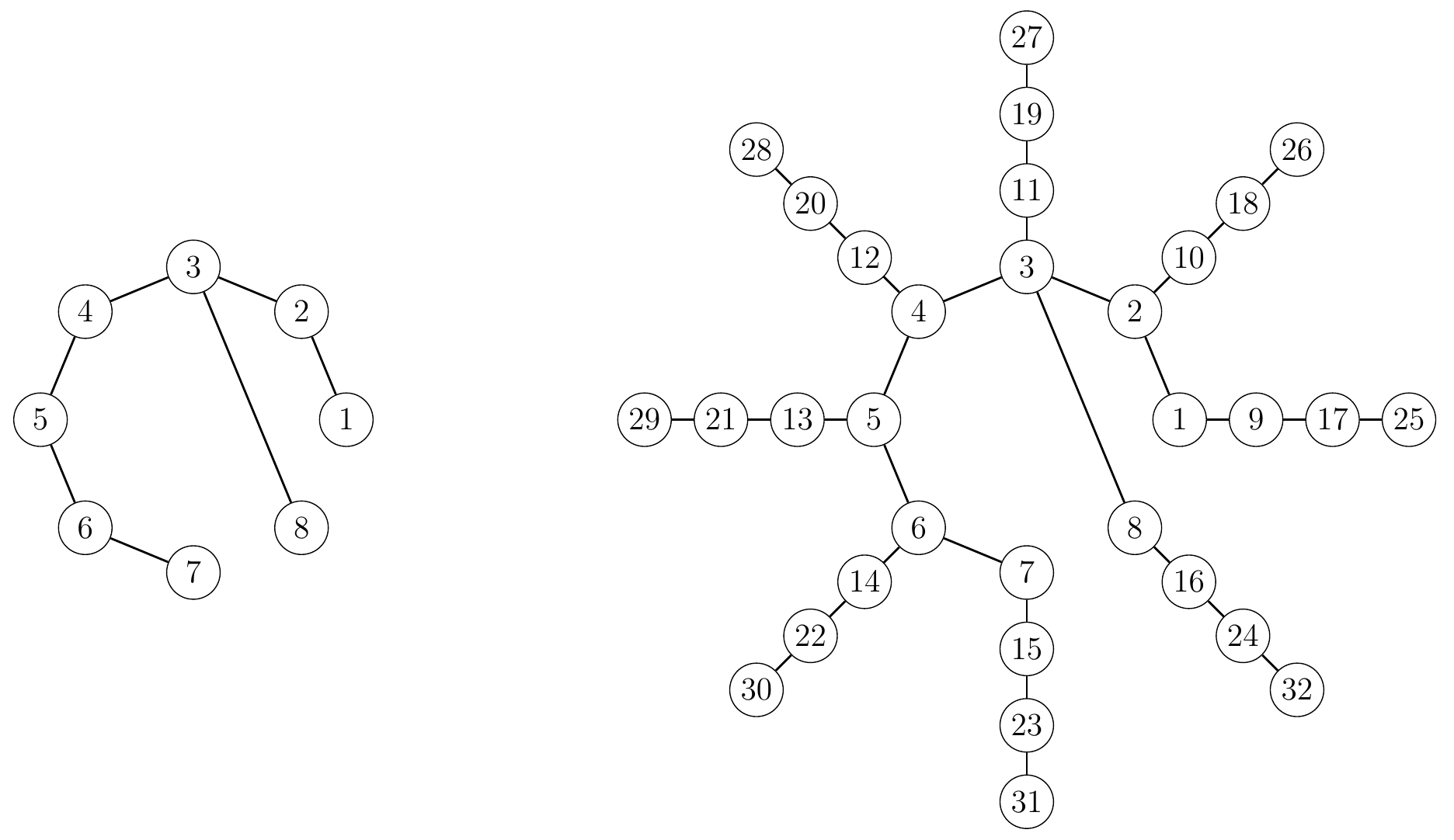}
	\caption{Graph $G$ (left) and the rooted product graph $G\circ P_4$ (right).}
\end{figure}
\begin{definition}\normalfont{
	Let $A=(a_{ij})$ be an $m\times n$ matrix and $B$  a $ p \times  q$ matrix. The \emph{Kronecker product} $A\otimes B$ is the $pm \times qn$ block matrix:
	$$A\otimes B=\begin{bmatrix}
	a_{11}B&\cdots&a_{1n}B\\
	\vdots&\ddots&\vdots\\
	a_{m1}B&\cdots&a_{mn}B
	\end{bmatrix}.
	$$
}
\end{definition}
By appropriately labelling the vertices in $G\circ P_m$, the adjacency matrix $A(G\circ P_m)$ has a nice structure.
\begin{observation}\label{adjGP}
	$A(G\circ P_m)=A(P_m)\otimes I_n+D_1\otimes A(G),$
	where $I_n$ is the identity matrix of order $n$ and $D_1$ is the diagonal matrix $\diag(1,0,\ldots,0)$ of order $m$.
	\end{observation}
For a graph $G$, we use $\phi(G;x)$ to denote the characteristic polynomial of $G$, i.e., $\phi(G;x)=\det(xI-A(G))$. The roots of $\phi(G;x)=0$ are called the eigenvalues (spectrum) of $G$.  The following lemma is a special case of decomposition of $\phi(G\circ H,x)$  derived by Schwenk \cite{schwenk1974} (see also \cite{godsil1978,gutman1980}).
\begin{lemma}
	$\phi(G\circ P_m;x)=(\phi(P_{m-1};x))^n\phi\left(G;\frac{\phi(P_m;x)}{\phi(P_{m-1};x)}\right)=\prod\limits_{i=1}^n(\phi(P_m;x)-\lambda_i\phi(P_{m-1};x)),$
	where $\lambda_1,\ldots,\lambda_n$ are eigenvalues of $G$.
\end{lemma}

Let $U_n(x)$ be the $n$-th Chebyshev polynomial of the second kind, defined by
$$U_n(\cos\theta)=\frac{\sin(n+1)\theta}{\sin\theta}.$$
It is known that $U_n(x)$'s satisfy the three-term recurrence relations:
$U_{n+1}(x)=2xU_n(x)-U_{n-1}(x)$ and the initial conditions:  $U_0(x)=1$ and $U_1(x)=2x$. Define $S_n(x)=U_n(x/2)$. Then $S_n(x)$ is a monic polynomial with integral coefficients and is referred to as the \emph{renormalized} Chebyshev polynomial \cite{rivlin1990}. It is well known that $\phi(P_m;x)=S_m(x)$.
\begin{definition}\label{eigmu}\normalfont{
Let $\lambda_1,\ldots,\lambda_n$ denote the eigenvalues of $G$ and $\xi_1,\ldots,\xi_n$ be the corresponding normalized eigenvector. 	We use $\mu_i^{(j)}(j\in\{1,2,\ldots,m\})$ to denote all zeroes of $S_m(x)-\lambda_iS_{m-1}(x)$ for  $i\in\{1,2,\ldots,n\}$ and write $$\eta_i^{(j)}=\frac{1}{S_{m-1}(\mu_i^{(j)})}\begin{bmatrix}
S_{m-1}(\mu_i^{(j)})\\
S_{m-2}(\mu_i^{(j)})\\
\vdots\\
S_0(\mu_i^{(j)})
\end{bmatrix}\otimes \xi_i.$$
}
\end{definition}
It should be pointed out that $S_{m-1}(\mu_i^{(j)})$ is never zero, see Corollary \ref{res1restated} in Sect.~\ref{sectres}.
\begin{lemma}\label{distmu}
For any $i\in\{1,2,\ldots,n\}$, all roots of $S_m(x)-\lambda_iS_{m-1}(x)$ are simple, i.e.,  $\mu_i^{(j_1)}\neq \mu_i^{(j_2)}$ for any distinct $j_1$ and $j_2$ in $\{1,2,\ldots,m\}$.
\end{lemma}
\begin{proof}
	The roots of the renormalized Chebyshev polynomials $S_m(x)$ and $S_{m-1}(x)$ are $\{a_k=2\cos\frac{k\pi}{m+1}\colon\,1\le k\le m\}$ and $\{b_k=2\cos\frac{k\pi}{m}\colon\,1\le k\le m-1\}$. Note that $a_1>b_1>a_2>b_2>\cdots>a_{m-1}>b_{m-1}>a_m$.
Moreover, as all roots $a_k$'s are clearly \emph{simple}, we find that  the sequence $S_m(+\infty), S_m(b_1),S_m(b_2),\ldots,S_m(b_{m-1}),S_m(-\infty)$ must have alternating signs. Write $f(x)=S_m(x)-\lambda_iS_{m-1}(x)$. Since $S_{m-1}(b_k)=0$ for $k=1,2,\ldots,m-1$ and  $m-1=\deg S_{m-1}(x)<\deg S_m(x)=m$, the signs of $S_m(x)$ and $f(x)$ are the same for each $x\in\{b_1,\ldots,b_{m-1}\}\cup\{+\infty,-\infty\}$. This means the sequence $f(+\infty)$, $f(b_1)$, $f(b_2)$, $\ldots$, $f(b_{m-1})$, $f(-\infty)$ also have alternating signs. By intermediate zero theorem for continuous functions, we see that $f(x)$ has at least one root in each of the $m$ intervals: $(-\infty,b_{m-1})$,  $(b_{m-1},b_{m-2})$, $\ldots,$ $(b_2,b_1)$ and $(b_1,+\infty)$. Since $f(x)$ is a polynomial of degree $m$, this means all roots of $f(x)$ are simple, as desired.
\end{proof}
\begin{lemma}\label{eigA}
	Let $\tilde{A}$ denote the adjacency matrix of $G\circ P_m$. Then $\tilde{A}\eta_i^{(j)}=\mu_i^{(j)}\eta_i^{(j)}$ for $i\in\{1,2,\ldots,n\}$ and $j\in\{1,2,\ldots,m\}$.
\end{lemma}
\begin{proof}
We fix $i$ and $j$ and write $s_k=S_k(\mu_i^{(j)})$ $(k=0,1,\ldots, m-1)$ for simplicity. By Observation \ref{adjGP} and some basic properties of the Kronecker product, we obtain
	\begin{eqnarray}\label{Aeta}
\tilde{A}\eta_i^{(j)} &=&\frac{1}{s_{m-1}}(A(P_m)\otimes I_n+D_1\otimes A(G))((s_{m-1},s_{m-2},\ldots,s_0)^\T\otimes \xi_i)\nonumber\\
&=&\frac{1}{s_{m-1}}((A(P_m)\begin{bmatrix}
s_{m-1}\\
s_{m-2}\\
\vdots\\
s_0
\end{bmatrix})\otimes \xi_i+\begin{bmatrix}
s_{m-1}\\
0\\
\vdots\\
0
\end{bmatrix}\otimes (\lambda_i \xi_i))\nonumber\\
&= &\frac{1}{s_{m-1}}(A(P_m)\begin{bmatrix}
s_{m-1}\\
s_{m-2}\\
\vdots\\
s_0
\end{bmatrix}+\begin{bmatrix}
\lambda_i s_{m-1}\\
0\\
\vdots\\
0
\end{bmatrix})\otimes \xi_i\nonumber\\
&= &\frac{1}{s_{m-1}}\begin{bmatrix}
s_{m-2}+\lambda_is_{m-1}\\
s_{m-3}+s_{m-1}\\
\vdots\\
s_0+s_2\\
s_1
\end{bmatrix}\otimes \xi_i.
\end{eqnarray}
By Definition \ref{eigmu}, we see that $\lambda_i s_{m-1}=s_m$. Noting that $S_1(x)=xS_0(x)$ and $S_k(x)+S_{k+2}(x)=xS_{k+1}$ for any $k\ge 0$, we obtain
\begin{equation}
\begin{bmatrix}
s_{m-2}+\lambda_is_{m-1}\\
s_{m-3}+s_{m-1}\\
\vdots\\
s_0+s_2\\
s_1
\end{bmatrix}=\begin{bmatrix}
s_{m-2}+s_{m}\\
s_{m-3}+s_{m-1}\\
\vdots\\
s_0+s_2\\
s_1
\end{bmatrix}=\mu_i^{(j)}\begin{bmatrix}
s_{m-1}\\
s_{m-2}\\
\vdots\\
s_1\\
s_0
\end{bmatrix}.
\end{equation}
Now, Eq. \eqref{Aeta} becomes
$$\tilde{A}\eta_i^{(j)}=\frac{1}{s_{m-1}}\mu_{i}^{(j)}\begin{bmatrix}
s_{m-1}\\
s_{m-2}\\
\vdots\\
s_1\\
s_0
\end{bmatrix}\otimes\xi_i=\mu_i^{(j)}\eta_i^{(j)},$$
as desired.
\end{proof}

\section{Resultants of Chebyshev-related polynomials}\label{sectres}
\begin{definition}\normalfont{
	Let $f(x)=a_nx^n+a_{n-1}x^{n-1}+\cdots+a_1x+a_0$ and $g(x)=b_mx^m+b_{m-1}x^{m-1}+\cdots+b_1x+b_0$. The resultant of $f(x)$ and $g(x)$, denoted by $\Res_x(f(x),g(x))$, or simply $\Res(f(x),g(x))$, is defined to be
	$$a_n^mb_m^n\prod_{1\le i\le n,1\le j\le m}(\alpha_i-\beta_j),$$
	where $\alpha_i$'s and $\beta_j$'s are the roots (in complex field $\mathbb{C}$) of $f(x)$ and $g(x)$, respectively.
}
\end{definition}
We list some basic properties of resultants for convenience.
\begin{lemma}\label{basicres}
	Let $f(x)=a_nx^n+\cdots+a_0=a_n\prod_{i=1}^n(x-\alpha_i)$ and $g(x)=b_mx^m+\cdots+b_0=b_m\prod_{j=1}^m(x-\beta_j)$. Then the followings hold:\\
	\textup{(\rmnum{1})} $\Res(f(x),g(x))=a_n^m\prod_{i=1}^{n}g(\alpha_i) =(-1)^{mn}b_m^{n}\prod_{j=1}^m f(\beta_j);$\\
	\textup{(\rmnum{2})} If $m<n$ then $\Res_x(f(x)+tg(x),g(x))=\Res_x(f(x),g(x))$ for any $t\in \mathbb{C}$;\\
	\textup{(\rmnum{3})} $\Res_x(f(tx),g(tx))=t^{mn}\Res_x(f(x),g(x))$ for any $t\in \mathbb{C}\setminus\{0\}.$
\end{lemma}
We need the following result due to Dilcher and Stolarsky \cite{dilcher2005}, see \cite{jacobs2011,louboutin2013} for different proofs.
\begin{lemma}[\cite{dilcher2005}]For any integer $m\ge 1$,
$$\Res(U_m(x),U_{m-1}(x))=(-1)^{\frac{m(m-1)}{2}}2^{m(m-1)}.$$
\end{lemma}
\begin{corollary}\label{res1restated}
	For any integers $m,n\ge 1$ and $i\in\{1,\ldots,n\}$,
	$$\prod_{j=1}^mS_{m-1}(\mu_i^{(j)})=(-1)^{\frac{m(m-1)}{2}}.$$	
\end{corollary}
\begin{proof}
	Recall that $S_{m-1}(x)=U_{m-1}(x/2)$ is a monic polynomial and so is $S_{m}(x)-\lambda_iS_{m-1}(x).$ Since $\mu_i^{(j)}$'s are roots of $S_{m}(x)-\lambda_iS_{m-1}(x)$, we obtain, by Lemma \ref{basicres},
		\begin{eqnarray*}\label{Res1}
\prod_{j=1}^mS_{m-1}(\mu_i^{(j)}) &=&\Res(S_{m}(x)-\lambda_iS_{m-1}(x),S_{m-1}(x))\\
&=&\Res(S_{m}(x),S_{m-1}(x))\\
&=&\left(\frac{1}{2}\right)^{m(m-1)}\Res(U_{m}(x),U_{m-1}(x))\\
&=&(-1)^{\frac{m(m-1)}{2}},
	\end{eqnarray*}
	as desired.
\end{proof}
The following result on the resultant of Chebyshev-related polynomials seems to be new. We guessed this equality during this study and asked for a proof on MathOverflow. We appreciate Terry Tao very much for his elegant proof which we include here.
\begin{lemma}[\cite{MO}]\label{newres}For any integer $m\ge 1$ and any complex number $t$,
	$$\Res_x \left( U_m(x)+tU_{m-1}(x),\sum_{k=0}^{m-1}U_k(x) \right) =(-1)^{\frac{m(m-1)}{2}} t^{\left\lfloor\frac{m}{2} \right\rfloor}2^{m(m-1)}.$$
\end{lemma}
\begin{proof}
	Since $U_m(x) + t U_{m-1}(x)$ is of degree $m$ and $\sum_{k=0}^{m-1} U_k(x)$ is of degree $m-1$ with leading coefficient $2^{m-1}$, the resultant factors as
	$$  (-1)^{m(m-1)}2^{m(m-1)} \prod_{j=1}^{m-1} (U_m(\beta_j) + t U_{m-1}(\beta_j))$$
	where $\beta_1,\dots,\beta_{m-1}$ are zeroes of $\sum_{k=0}^{m-1} U_k(x)$.
	
	Fortunately, these zeroes can be located explicitly using the usual trigonometric addition and subtraction identities. Telescoping the trigonometric identity $$\sin k \theta = \frac{\cos\left(k-\frac{1}{2}\right) \theta - \cos\left(k+\frac{1}{2}\right) \theta}{2 \sin \frac{\theta}{2} }$$
	we conclude that
	$$ \sum_{k=0}^{m-1} U_k(\cos \theta) = \frac{1}{\sin\theta} \sum_{k=1}^{m}\sin k\theta=\frac{\cos\frac{\theta}{2}-\cos\left(m+\frac{1}{2}\right)\theta}{2 \sin \theta \sin \frac{\theta}{2}} = \frac{ \sin\frac{m+1}{2} \theta\sin\frac{m}{2} \theta}{2 \cos \frac{\theta}{2} \sin^2 \frac{\theta}{2}}$$
	and so the $m-1 = \lfloor \frac{m}{2} \rfloor + \lfloor \frac{m-1}{2} \rfloor$ zeroes of $\sum_{k=0}^{m-1} U_k(x)$ take the form $\cos\frac{2\pi j}{m+1}$ for $1 \leq j \le \lfloor\frac{m}{2}\rfloor$ and $\cos \frac{2\pi j}{m} $ for $1 \leq j \le \lfloor\frac{m-1}{2}\rfloor$.
	
	Since the first class $\cos \frac{2\pi j}{m+1} $ of zeroes are also zeroes of $U_m(x)$, and the second class $\cos\frac{2\pi j}{m}$ are zeroes of $U_{m-1}(x)$, the resultant therefore simplifies to
	$$  (-1)^{m(m-1)}2^{m(m-1)} t^{\lfloor \frac{m}{2} \rfloor}
	\prod_{1 \leq j \le\lfloor \frac{m}{2}\rfloor} U_{m-1}\left( \cos\frac{2\pi j}{m+1} \right)
	\prod_{1 \leq j \le\lfloor\frac{m-1}{2}\rfloor} U_{m}\left( \cos\frac{2\pi j}{m} \right).$$
	But
	$$ U_{m-1}\left( \cos\frac{2\pi j}{m+1}  \right) =\frac{\sin \frac{2\pi mj}{m+1}}{\sin\frac{2\pi j}{m+1}}  = -1$$
	and similarly
	$$ U_{m}\left( \cos \frac{2\pi j}{m}  \right) = \frac{\sin\frac{2\pi (m+1)j}{m}}{\sin\frac{2\pi j}{m}}  = +1$$
	and the lemma then follows after counting up the signs.
\end{proof}
Using a similar argument as in the proof of Corollary \ref{res1restated}, we obtain the following
\begin{corollary}\label{res2restated}
		For any integers $m,n\ge 1$ and $i\in\{1,\ldots,n\}$,
	$$\prod_{j=1}^m\sum_{k=0}^{m-1}S_k(\mu_i^{(j)})=(-1)^{\frac{m(m-1)}{2}} (-\lambda_i)^{\left\lfloor\frac{m}{2} \right\rfloor}.$$
\end{corollary}
\section{Proof of Theorem \ref{main}}
\begin{lemma}[\cite{mao2015}]\label{basicW}
	Let $\lambda_i$ be the eigenvalues of $G$ with normalized eigenvector $\xi_i$ for $i=1,2,\ldots,n$. Then
	$$\det W(G)=\pm \prod_{1\le i_1\le i_2\le n}(\lambda_{i_2}-\lambda_{i_1})\prod_{1\le i\le n}(e_n^\T \xi_i).$$
\end{lemma}
\begin{definition}
	$S(x)=\sum_{k=0}^{m-1}S_k(x).$
\end{definition}
The following equality is straightforward.
\begin{lemma}\label{emn}
	$e_{mn}^\T \eta_i^{(j)}=\frac{S(\mu_{i}^{(j)})}{S_{m-1}(\mu_i^{(j)})}e_n^\T \xi_i.$
\end{lemma}
Let $\tilde{A}=A(G\circ P_m)$. By Lemmas \ref{eigA} and \ref{emn}, we have
	\begin{eqnarray}\label{eA}
&&e_{mn}^\T\tilde{A}^k[\eta_1^{(1)},\ldots,\eta_n^{(1)};\ldots;\eta_1^{(m)},\ldots,\eta_n^{(m)}]\nonumber\\
&=&e_{mn}^\T[\eta_1^{(1)},\ldots,\eta_n^{(1)};\ldots;\eta_1^{(m)},\ldots,\eta_n^{(m)}]\begin{bmatrix}
(\mu_1^{(1)})^k&&&\\
&(\mu_2^{(1)})^k&&\\
&&\ddots&\\
&&&(\mu_n^{(m)})^k
\end{bmatrix}_{(mn)\times (mn)}\nonumber\\
&=&[(\mu_1^{(1)})^k,(\mu_2^{(1)})^k,\ldots,(\mu_n^{(m)})^k]\begin{bmatrix}
\frac{S(\mu_{1}^{(1)})}{S_{m-1}(\mu_1^{(1)})}e_n^\T \xi_1&&&\\
&\frac{S(\mu_{2}^{(1)})}{S_{m-1}(\mu_2^{(1)})}e_n^\T \xi_2&&\\
&&\ddots&\\
&&&\frac{S(\mu_{n}^{(m)})}{S_{m-1}(\mu_n^{(m)})}e_n^\T \xi_n
\end{bmatrix}.
\end{eqnarray}
Let $D$ denote the diagonal matrix in Eq.~\eqref{eA}. Write $E^{(j)}=[\eta_1^{(j)},\ldots,\eta_n^{(j)}]$ and
$$M^{(j)}=\begin{bmatrix}
1&1&\cdots&1\\
\mu_1^{(j)}&\mu_2^{(j)}&\cdots&\mu_n^{(j)}\\
\vdots&\vdots&\vdots&\vdots\\
(\mu_1^{(j)})^{mn-1}&(\mu_2^{(j)})^{mn-1}&\cdots&(\mu_n^{(j)})^{mn-1}
\end{bmatrix}_{(mn)\times n}.
$$
Then, summarizing Eq. \eqref{eA} for $k$ from $0$ to $mn-1$ leads to
\begin{equation}\label{wemd}
(W(G\circ P_m))^\T[E^{(1)},\ldots,E^{(m)}]=[M^{(1)},\ldots,M^{(m)}]D.
\end{equation}

\noindent\textbf{Claim 1}: It holds that
$$\det D=a_0^{\lfloor\frac{m}{2}\rfloor}\left(\prod\limits_{1\le i\le n}e_n^\T\xi_i\right)^m,$$ where $a_0=(-1)^n\det A(G)$,  the constant term of characteristic polynomial of $G$.

 By Corollary \ref{res1restated}, we obtain
 $$\prod_{1\le i\le n}\prod_{1\le j\le m}S_{m-1}(\mu_i^{(j)})=(-1)^\frac{m(m-1)n}{2}.$$
 Recall that $S(x)=\sum_{k=0}^{m-1}S_k(x)$. We can rewrite Corollary \ref{res2restated} as
$$\prod_{j=1}^mS(\mu_i^{(j)})=(-1)^{\frac{m(m-1)}{2}} (-\lambda_i)^{\left\lfloor\frac{m}{2} \right\rfloor}.$$
Noting that $\prod_{i=1}^{n}(-\lambda_i)=a_0$, we obtain
$$\prod_{i=1}^{n}\prod_{j=1}^mS(\mu_i^{(j)})=(-1)^{\frac{m(m-1)n}{2}} a_0^{\lfloor\frac{m}{2} \rfloor}.$$
and hence Claim 1 follows by direct multiplication of all diagonal entries in $D$.

\noindent\textbf{Claim 2}: It holds that
$$\det [M^{(1)},\ldots,M^{(m)}]=\left(\prod_{i=1}^{n}\prod_{1\le j_1<j_2\le m}\left(\mu_i^{(j_2)}-\mu_i^{(j_1)}\right)\right)\left(\prod_{1\le i_1< i_2\le n}(\lambda_{i_2}-\lambda_{i_1})\right)^m.$$

We use co-lexicographical order for Cartesian product $\{1,\ldots,n\}\times \{1,\ldots,m\}$. Precisely, $(i_1,j_1)<(i_2,j_2)$ (in co-lexicographical order) if either $j_1<j_2$,  or $j_1=j_2$ and $i_1<i_2$.

Write $V=[M^{(1)},\ldots,M^{(m)}]$. Using the familiar formula for a Vandermonde matrix, we obtain
 	\begin{eqnarray}\label{detV}
 \det V\nonumber&=&\prod_{(i_1,j_1)<(i_2,j_2)}\left(\mu_{i_2}^{(j_2)}-\mu_{i_1}^{(j_1)}\right)\\
 &=&\left(\prod_{i=1}^{n}\prod_{1\le j_1<j_2\le m}\left(\mu_{i}^{(j_2)}-\mu_{i}^{(j_1)}\right)\right)\left(\prod_{i_1\neq i_2}\prod_{(i_1,j_1)<(i_2,j_2)}\left(\mu_{i_2}^{(j_2)}-\mu_{i_1}^{(j_1)}\right)\right).
 \end{eqnarray}
 The second factor of Eq.~\eqref{detV} can be regrouped as
 	\begin{eqnarray}\label{detV2}
&&\prod_{1\le i_1< i_2\le n}\left(\prod_{(i_1,j_1)<(i_2,j_2)}\left(\mu_{i_2}^{(j_2)}-\mu_{i_1}^{(j_1)}\right)\right)\left(\prod_{(i_2,j_2)<(i_1,j_1)}\left(\mu_{i_1}^{(j_1)}-\mu_{i_2}^{(j_2)}\right)\right)\nonumber\\
 &=&\prod_{1\le i_1< i_2\le n}\left(\prod_{j_2=1}^{m}\prod_{j_1=1}^{m}\left(\mu_{i_2}^{(j_2)}-\mu_{i_1}^{(j_1)}\right)\right)\left(\prod_{(i_2,j_2)<(i_1,j_1)}(-1)\right).
 \end{eqnarray}
 Recall that $S_m(x)-\lambda_iS_{m-1}(x)=\prod_{j=1}^m(x-\mu_i^{(j)})$. We see that
 $$\prod_{j_2=1}^{m}\prod_{j_1=1}^{m}\left(\mu_{i_2}^{(j_2)}-\mu_{i_1}^{(j_1)}\right)=\prod_{j_2=1}^{m}\left(S_m(\mu_{i_2}^{(j_2)})-\lambda_{i_1}S_{m-1}(\mu_{i_2}^{(j_2)})\right)=(\lambda_{i_2}-\lambda_{i_1})^m\prod_{j_2=1}^{m}S_{m-1}(\mu_{i_2}^{(j_2)}).$$

 Note that for $i_1< i_2$, the inequality $(i_2,j_2)<(i_1,j_1)$ holds if and only if $j_2<j_1$. This means, for any fixed $i_1,i_2\in\{1,2,\ldots,n\}$ with $i_1<i_2$, $$\prod_{(i_2,j_2)<(i_1,j_1)}(-1)=(-1)^\frac{m(m-1)}{2}.$$ Finally, by Corollary \ref{res1restated},
 $$\prod_{j_2=1}^{m}S_{m-1}(\mu_{i_2}^{(j_2)})=(-1)^\frac{m(m-1)}{2}.$$
 Now, Eq.~\eqref{detV2} reduces to
 $$\left(\prod_{1\le i_1< i_2\le n}\left(\lambda_{i_2}-\lambda_{i_1}\right)\right)^m$$
 and Claim 2 holds by Eq. \eqref{detV}.

 \noindent\textbf{Claim 3}: It holds that
 $$\det[E^{(1)},\ldots,E^{(m)}]=\pm \prod_{i=1}^{n}\prod_{1\le j_1< j_2\le m}\left(\mu_i^{(j_2)}-\mu_i^{(j_1)}\right).$$

 Let $\tilde{s}_k^{(i,j)}=\frac{S_k(\mu_i^{(j)})}{S_{m-1}(\mu_i^{(j)})}$ and $K_{j_1,j_2}$ be the diagonal  matrix:
 \begin{equation}
K_{j_1,j_2}=\begin{bmatrix}
\tilde{s}_{m-j_1}^{(1,j_2)}&&&\\
&\tilde{s}_{m-j_1}^{(2,j_2)}&&\\
&&\ddots&\\
&&&\tilde{s}_{m-j_1}^{(n,j_2)}
\end{bmatrix}_{n\times n}.
 \end{equation}
 Write $Q=[\xi_1,\xi_2,\ldots,\xi_n]$.  Then it is routine to check the following factorization:
 \begin{equation}\label{efac}
 [E^{(1)},E^{(2)},\ldots,E^{(m)}]=\begin{bmatrix}
Q&&&\\
&Q&&\\
&&\ddots&\\
&&&Q
	\end{bmatrix} \begin{bmatrix}
	K_{1,1}&K_{1,2}&\cdots&K_{1,m}\\
	K_{2,1}&K_{2,2}&\cdots&K_{2,m}\\
	\cdots&\cdots&\cdots&\cdots\\
	K_{m,1}&K_{m,2}&\cdots&K_{m,m}
	\end{bmatrix},
 \end{equation}
 where the first factor is a block diagonal matrix with $m$ identical diagonal blocks. Since each block $K_{j_1,j_2}$ is diagonal, it is not difficult to see that, via an  appropriate permutation matrix,  the block matrix $[K_{j_1,j_2}]_{m\times m}$ is similar to the following block diagonal matrix
 \begin{equation*}
 \begin{bmatrix}
L_{1}&&&\\
&L_2&&\\
&&\ddots&\\
&&&L_n
 \end{bmatrix},
 \end{equation*}
 where the $(j_1,j_2)$-th entry of each $L_i$ is the $(i,i)$-th entry of $K_{j_1,j_2}$.  Written exactly,
 \begin{eqnarray}\label{Lt}
 L_i&=&\begin{bmatrix}
 \tilde{s}_{m-1}^{(i,1)}& \tilde{s}_{m-1}^{(i,2)}&\ldots& \tilde{s}_{m-1}^{(i,m)}\\
  \tilde{s}_{m-2}^{(i,1)}& \tilde{s}_{m-2}^{(i,2)}&\ldots& \tilde{s}_{m-2}^{(i,m)}\\
  \cdots&\cdots&\cdots&\cdots\\
   \tilde{s}_{0}^{(i,1)}& \tilde{s}_{0}^{(i,2)}&\ldots& \tilde{s}_{0}^{(i,m)}\\
 \end{bmatrix}\nonumber\\
 &=&\begin{bmatrix}
S_{m-1}(\mu_i^{(1)})&\ldots&  S_{m-1}(\mu_i^{(m)})\\
S_{m-2}(\mu_i^{(1)})&\ldots&  S_{m-2}(\mu_i^{(m)})\\
 \cdots&\cdots&\cdots\\
S_{0}(\mu_i^{(1)})&\ldots&  S_{0}(\mu_i^{(m)})\\
 \end{bmatrix}\begin{bmatrix}
\frac{1}{S_{m-1}(\mu_i^{(1)})}&&\\
 &\ddots&\\
&& \frac{1}{S_{m-1}(\mu_i^{(m)})}\\
 \end{bmatrix}.
 \end{eqnarray}
 Since $S_k(x)$ is a monic polynomial with degree $k$ for each nonnegative integer $k$, we find that the determinant of the first factor in Eq.~\eqref{Lt} equals
 \begin{equation}\label{detL}
 \det\begin{bmatrix}
 (\mu_i^{(1)})^{m-1}&(\mu_i^{(2)})^{m-1}&\cdots& (\mu_i^{(m)})^{m-1}\\
 \cdots&\cdots&\cdots&\cdots\\
 \mu_i^{(1)}&\mu_i^{(2)}&\cdots& \mu_i^{(m)}\\
1&1& \cdots&1
 \end{bmatrix}=(-1)^{\frac{m(m-1)}{2}}\prod_{1\le j_1< j_2\le m}\left(\mu_i^{(j_2)}-\mu_i^{(j_1)}\right).
 \end{equation}
 As $\det Q=\pm 1$ and $\det [K_{j_1,j_2}]_{m\times m}=\prod_{i=1}^n\det L_i$, Claim 3 follows by Eqs.~\eqref{efac}-\eqref{detL}.

 By Claim 3 and Lemma \ref{distmu}, we see that  $\det[E^{(1)},\ldots,E^{(m)}]\neq 0$. Taking determinants for both sides of Eq.~\eqref{wemd} and using Claims 1-3, we obtain
 \begin{eqnarray*}
\det W(G\circ P_m)&=&\frac{\det[M^{(1)},\ldots,M^{(m)}]\det D}{\det [E^{(1)},\ldots,E^{(m)}]}\\
&=&\pm a_0^{\lfloor\frac{m}{2}\rfloor}\left(\prod_{1\le i_1<i_2\le n}(\lambda_{i_2}-\lambda_{i_1})\right)^m\left(\prod\limits_{1\le i\le n}e_n^\T\xi_i\right)^m\\
&=&\pm a_0^{\lfloor\frac{m}{2}\rfloor}(\det W(G))^m,
 \end{eqnarray*}
where the last equality follows from Lemma \ref{basicW}. This completes the proof of Theorem \ref{main}.

For any even positive integer $n$, let $\mathcal{F}^*_{n}$ be the family of all $n$-vertex graph $G$ such that $ \det W(G)=\pm 2^\frac{n}{2}$ and the constant term of $\phi(G;x)$ is $\pm 1$. Write $$\mathcal{F}^*=\bigcup_{n \textup{~even}}\mathcal{F}^*_{n}.$$
As a special case of the aforementioned theorem of Wang \cite{wang2017JCTB}, each graph in $\mathcal{F}^*$ is DGS (determined by its generalized spectrum).
\begin{lemma}[\cite{mao2022}]\label{constant1}
	If the constant term of $\phi(G;x)$ is $\pm 1$, then so is  $\phi(G\circ P_m;x)$ for each integer $m\ge 2$.
\end{lemma}
As a direct consequence of Theorem \ref{main} and Lemma \ref{constant1}, we obtain the following result which was  conjectured (in a slightly different form) by Mao and Wang \cite[Conjecture 3.1]{mao2022}.
\begin{theorem}\label{dgs}	If $G\in \mathcal{F}^*$ then for any integer $m\ge 2$, the graph $G\circ P_m\in \mathcal{F}^*$ and hence is DGS.
\end{theorem}
Theorem \ref{dgs} gives a simple method to construct large DGS-graphs from small ones. For example, let $G$ be the left graph in Figure 1. It can be easily checked that $G\in \mathcal{F}^*$. Thus using Theorem \ref{dgs}  iteratively, we see that, for any  positive sequence $m_1,m_2,\ldots,m_k,\ldots,$ with $m_i\ge 2$, all graphs in the family
$$G\circ P_{m_1},(G\circ P_{m_1})\circ P_{m_2},((G\circ P_{m_1})\circ P_{m_2})\circ P_{m_3},\cdots$$
are DGS.

\section*{Acknowledgments}
This work is supported by the	National Natural Science Foundation of China (Grant Nos. 12001006 and 12101379), Natural Science Basic Research Plan in Shannxi Province of China (Grant No.\,2020JQ-696) and the Scientific Research Foundation of Anhui Polytechnic University (Grant No.\,2019YQQ024).

\end{document}